\documentclass[a4paper]{amsart}
\usepackage{footmisc}
\usepackage{amssymb, enumerate, mathtools}
\usepackage{hyperref}

\renewcommand{\AA}{\mathbf{A}}
\newcommand{\QQ}{\mathbf{Q}}
\newcommand{\CC}{\mathbf{C}}
\newcommand{\ZZ}{\mathbf{Z}}
\newcommand{\RR}{\mathbf{R}}

\newcommand{\Zp}{\ZZ_p}
\newcommand{\Qp}{\QQ_p}

\newcommand{\fl}{\mathfrak{l}}
\newcommand{\fg}{\mathfrak{g}}

\newcommand{\fp}{\mathfrak{p}}
\newcommand{\fpb}{{\overline{\fp}}}
\newcommand{\fq}{\mathfrak{q}}
\newcommand{\ff}{\mathfrak{f}}
\newcommand{\fm}{\mathfrak{m}}
\newcommand{\fF}{\mathfrak{F}}
\newcommand{\ord}{\mathrm{ord}}
\newcommand{\la}{\mathrm{la}}
\newcommand{\cO}{\mathcal{O}}
\newcommand{\cE}{\mathcal{E}}
\newcommand{\into}{\hookrightarrow}

\newcommand{\cG}{\mathcal{G}}
\newcommand{\cH}{\mathcal{H}}

\DeclareMathOperator{\GL}{GL}
\DeclareMathOperator{\Lie}{Lie}
\DeclareMathOperator{\Tr}{Tr}
\DeclareMathOperator{\Norm}{Norm}

\newtheorem{theorem}{Theorem}[section]

\newtheorem{proposition}[theorem]{Proposition}
\newtheorem{corollary}[theorem]{Corollary}
\newtheorem{definition}[theorem]{Definition}
\theoremstyle{remark}
\newtheorem{remark}[theorem]{Remark}
\newtheorem{example}[theorem]{Example}

\begin{document}

 \title[Integration on ray class groups]{P-adic integration on ray class groups and non-ordinary p-adic L-functions}
 \author{David Loeffler}
 \address{Mathematics Institute, University of Warwick, Coventry CV4 7AL, UK} \email{d.a.loeffler@warwick.ac.uk}

\thanks{Supported by a Royal Society University Research Fellowship.}

 \begin{abstract}
 We study the theory of $p$-adic finite-order functions and distributions on ray class groups of number fields, and apply this to the construction of (possibly unbounded) $p$-adic $L$-functions for automorphic forms on $\GL_2$ which may be non-ordinary at the primes above $p$. As a consequence, we obtain a ``plus-minus'' decomposition of the $p$-adic $L$-functions of automorphic forms for $\GL_2$ over an imaginary quadratic field with $p$ split and Hecke eigenvalues 0 at the primes above $p$, confirming a conjecture of B.D.~Kim.
 \end{abstract}

\maketitle

 \section{Introduction}

  $P$-adic $L$-functions attached to various classes of automorphic forms over number fields are an important object of study in Iwasawa theory. In most cases, these $p$-adic $L$-functions are constructed by interpolating the algebraic parts of critical values of classical (complex-analytic) $L$-functions of twists of the automorphic form by finite-order Hecke characters of the number field.

  When the underlying automorphic form is ordinary at $p$, or when the number field is $\QQ$, this $p$-adic interpolation is very well-understood. However, the case of non-ordinary automorphic forms for number fields $K \ne \QQ$ has received comparatively little study. This paper grew out of the author's attempts to understand the work of Kim \cite{kim-preprint}, who considered the $L$-functions of weight 2 elliptic modular forms over imaginary quadratic fields.

  In this paper, we develop a systematic theory of finite-order distributions and $p$-adic interpolation on ray class groups of number fields. Developing such a theory is a somewhat delicate issue, and appears to have been the subject of several mistakes in the literature to date; we hope this paper will go some way towards resolving the confusion. We then apply this theory to the study of $p$-adic $L$-functions for automorphic representations of $\GL_2$ over a general number field, extending the work of Shai Haran \cite{haran87} in the ordinary case. Largely for simplicity we assume the automorphic representations have trivial central character and lowest cohomological infinity-type (as in the case of representations attached to modular elliptic curves over $K$).

  Applying our theory to the case of $K$ imaginary quadratic with $p$ split in $K$, we obtain proofs of two conjectures. Firstly, we prove a special case of a conjecture advanced by the author and Zerbes in \cite{loefflerzerbes11}, confirming the existence of two ``extra'' $p$-adic $L$-functions whose existence was predicted on the basis of general conjectures on Euler systems. Secondly, we consider the case when the Hecke eigenvalues at both primes above $p$ are 0 (the ``most supersingular'' case); in this case Kim has predicted a decomposition of the $L$-functions analogous to the ``signed $L$-functions'' of Pollack \cite{pollack03} for modular forms. Using the additional information arising from our two extra $L$-functions, we prove Kim's conjecture, constructing four bounded $L$-functions depending on a choice of sign at each of the two primes above $p$.

  \subsection*{Acknlowedgements} I would like to thank B.D.~Kim for sending me an early draft of his paper \cite{kim-preprint}, together with many informative remarks on its contents. I am also grateful to Matthew Emerton, G\"unter Harder, Mladen Dimitrov and Sarah Zerbes for interesting and helpful conversations on this topic; and to the anonymous referee for several valuable comments and corrections.

 \section{Integration on $p$-adic groups}
  \label{sect:padicanalysis}

  In this section, we shall recall and slightly generalize some results concerning the space of locally analytic distributions on an abelian $p$-adic analytic group (the dual space of locally analytic functions); in particular, we are interested in when it is possible to uniquely interpolate a linear functional on locally constant functions by a locally analytic distribution satisfying some growth property.

  In this section we fix a prime $p$ and a coefficient field $E$, which will be a complete discretely-valued subfield of $\CC_p$, endowed with the valuation $v_p$ such that $v_p(p) = 1$.

  \subsection{One-variable theory}

   We first recall the well-known theory of finite-order functions and distributions on the group $\Zp$.

   For $r \in \RR_{\ge 0}$, we define the space of \emph{order $r$ functions} $C^r(\Zp, E)$ as the space of functions $\Zp \to E$ admitting a local Taylor expansion of degree $\lfloor r \rfloor$ at every point with error term $o(z^r)$, cf.~\cite[\S I.5]{colmez10}. We write $D^r(\Zp, E)$, the space of \emph{order $r$ distributions}, for the dual of $C^r(\Zp, E)$ (the space of linear functionals $C^r(\Zp, E) \to E$ which are continuous with respect to the Banach space topology of $C^r(\Zp, E)$; see \emph{op.cit.}~for the definition of this topology).

   We may regard the space $C^{\la}(\Zp, E)$ of locally analytic $E$-valued functions on $\Zp$ as a dense subspace of $C^r(\Zp, E)$ for any $r$, so dually all of the spaces $D^{r}(\Zp, E)$ may be regarded as subspaces of the $E$-algebra $D^{\la}(\Zp, E)$ of \emph{locally analytic distributions}, the dual of $C^{\la}(\Zp, E)$. It is well known that $D^{\la}(\Zp, E)$ can be interpreted as the algebra of functions on a rigid-analytic space over $E$ (isomorphic to the open unit disc), whose points parametrize continuous characters of $\Zp$.

   It is clear that the above definitions extend naturally if $\Zp$ is replaced by any abelian $p$-adic analytic group of dimension 1, since any such group has a finite-index open subgroup $H$ isomorphic to $\Zp$; we say a function, distribution, etc has a given property if its pullback under the map $\Zp \cong a \cdot H \subseteq G$ has it for every $a \in G / H$.

   We define a \emph{locally constant distribution} to be a linear functional on the space $LC(G, E)$ of locally constant functions on $\Zp$. This is clearly equivalent to the data of a finitely-additive $E$-valued function on open subsets of $G$, i.e.~a ``distribution'' in the sense\footnote{To avoid confusion we shall not use the term ``distribution'' alone in this paper, but rather in company with an adjectival phrase, such as ``locally analytic'', and generally an ``$X$ distribution'' shall mean ``a continuous linear functional on the space of $X$ functions'' (with respect to some topology to be understood from the context).} of some textbooks such as \cite{washington97}.

   \begin{definition}
    \label{def:1vargrowth}
    Let $G$ be an abelian $p$-adic analytic group of dimension 1, and let $\mu$ be an $E$-valued locally constant distribution on $G$. We say $\mu$ has \emph{growth bounded by $r$} if there exists $C \in \RR$ such that
    \[ \inf_{a \in G} v_p\, \mu\left(1_{a + p^m H} \right) \ge C - r m\]
    for all $m \ge 0$, where $H$ is some choice of subgroup of $G$ isomorphic to $\Zp$.
   \end{definition}

   \begin{remark}
    The definition is independent of the choice of $H$, although the constant $C$ may not be so.
   \end{remark}

   \begin{theorem}[Amice--V\'elu, Vishik]
    \label{thm:avv}
     Let $G$ be an abelian $p$-adic Lie group of dimension 1.
     \begin{enumerate}[(i)]
      \item If $\mu \in D^r(G, E)$, then the restriction of $\mu$ to $LC(G, E)$ has growth bounded by $r$.
      \item Conversely, if $r < 1$ and $\mu$ is a locally constant distribution on $G$ with growth bounded by $r$, there exists a unique element $\tilde\mu \in D^r(G, E)$ whose restriction to $LC(G, E)$ is $\mu$.
     \end{enumerate}
   \end{theorem}

   \begin{proof}
    This is a special case of \cite[Theor\`eme II.3.2]{colmez10}.
   \end{proof}

   \begin{remark}
    In the case $r \ge 1$, one can check that any locally constant distribution $\mu$ with growth bounded by $r$ can be extended to an element $\tilde\mu \in  D^r(G, E)$, but this is only uniquely defined modulo $\ell_0 D^{r - 1}(G, E)$, where $\ell_0$ is the order 1 distribution whose action on $C^1$ functions is $f \mapsto f'(0)$.
   \end{remark}

  \subsection{The case of several variables}
   \label{sect:severalvars1}

   We now consider the case of functions of several variables. First we consider the group $G = \Zp^d$, for some integer $d \ge 1$. Let $r$ be a fixed real number; for simplicity, we shall assume that $r < 1$.

   \begin{definition}
    Let $d \ge 1$ and let $f: G \to E$ be any continuous function. For $m \ge 0$ define the quantity $\varepsilon_m(f)$ by
    \[ \varepsilon_m(f) = \inf_{x \in G, y \in p^m G} v_p\left[f(x + y) - f(x)\right].\]
    We say $f$ has \emph{order $r$} if $\varepsilon_m(f) - rm \to \infty$ as $m \to \infty$.
   \end{definition}

   The space $C^r(G, E)$ of functions with this property is evidently an $E$-Banach space with the valuation
   \[ v_{C^r}(f) = \inf\left( \inf_{x \in G} v_p(f(x)), \inf_{m \ge 0} (\varepsilon_m(f) - rm) \right).\]

   \begin{remark}
    For $d = 1$ this reduces to the definition of the valuation denoted by $v'_{C^r}$ in \cite{colmez10}; in \emph{op.cit.} the notation $v_{C^r}$ is used for another slightly different valuation on $C^r(\Zp, E)$, which is not equal to $v'_{C^r}$ but induces the same topology.
   \end{remark}

   If we define, for any $f \in C^r(G, E)$, a sequence of functions $(f_m)_{m \ge 0}$ by
   \[ f_m = \sum_{j_1 = 0}^{p^{m} - 1}\dots \sum_{j_d = 0}^{p^{m} - 1} f(j_1, \dots, j_d) 1_{(j_1, \dots, j_d) + p^m G},\]
   then it is clear from the definition of the valuation $v_{C^r}$ that we have $f_m \to f(x)$ in the topology of $C^r(G, E)$, so in particular the space $LC(G, E)$ of locally constant $E$-valued functions on $G$ is dense in $C^r(G, E)$.

   We define a space $D^r(G, E)$ as the dual of $C^r(G, E)$, as before.

   As in the case $d = 1$ all of these constructions are clearly local on $G$ and thus extend\footnote{A little care is required here since for $k \ge 1$ pullback along the inclusion $\Zp \stackrel{\times p^k}{\longrightarrow} \Zp$ is a continuous map $C^r(\Zp, E) \to C^r(\Zp, E)$ but \emph{not} an isometry if $r \ne 0$. So for a general $G$ the space $C^r(G, E)$ has a canonical topology, but not a canonical valuation inducing this topology.} to abelian groups having an open subgroup isomorphic to $\Zp^d$. Since any abelian $p$-adic analytic group has this property for some $d \ge 0$, we obtain well-defined spaces $C^r(G, E)$ and dually $D^r(G, E)$ for any such group $G$.

   \begin{remark}
    This includes the case where $G$ is finite, so $d = 0$; in this case we clearly have have $C^r(G, E) = C^{\la}(G, E) = \operatorname{Maps}(G, E)$ and $D^r(G, E) = D^{\la}(G, E) = E[G]$ for any $r \ge 0$.
   \end{remark}

   We now consider the problem of interpolating locally constant distributions by order $r$ distributions.

   \begin{definition}
    \label{def:severalvarsgrowth}
    Let $\mu$ be a locally constant distribution on $G$. We say that $\mu$ has \emph{growth bounded by $r$} if there is a constant $C$ and an open subgroup $H$ of $G$ isomorphic to $\Zp^d$ such that
    \[ \inf_{a \in G} v_p \mu\left(1_{a + p^m H} \right) \ge C - rm \]
    for all $m \ge 0$.
   \end{definition}

   The natural analogue of the theorem of Amice--V\'elu--Vishik is the following.

   \begin{theorem}
    Let $G$ be an abelian $p$-adic analytic group, $r \in [0, 1)$, and $\mu$ a locally constant distribution on $G$ with growth bounded by $r$. Then there is a unique element $\tilde\mu \in D^r(G, E)$ whose restriction to $LC(\Zp, E)$ is $\mu$.
   \end{theorem}

   \begin{proof}
    The proof of this result in the general case is very similar to the case $d = 1$, so we shall give only a brief sketch. Clearly one may assume $G = H = \Zp^d$. One first shows that the space of locally constant functions on $\Zp^d$ has a ``wavelet basis'' consisting of the indicator functions of the sets
    \[ S(i_1, \dots, i_d) = (i_1 + p^{\ell(i_1)} \Zp) \times \dots \times (i_d + p^{\ell(i_d)} \Zp)\]
    for $i_1, \dots, i_d \ge 0$, where $\ell(n)$ is as defined in \cite[\S 1.3]{colmez10}. One then checks that the functions
    \[ e_{(i_1, \dots, i_d), r} := p^{\sup_j \left\lfloor r \ell(i_j) \right\rfloor} 1_{S(i_1, \dots, i_d)}\]
    form a Banach basis of $C^r(\Zp^d, E)$, so a linear functional on locally constant functions taking bounded values on the $e_{(i_1, \dots, i_d), r}$ extends uniquely to the whole of $C^r(\Zp^d, E)$.
   \end{proof}

   \begin{remark}
    The case when $G = \cO_L$, for some finite extension $L / \Qp$, has been studied independently by De Ieso \cite{deieso13}, who shows the following more general result: a locally polynomial distribution of degree $k$ which has growth of order $r$ in a sense generalizing Definition \ref{def:severalvarsgrowth} admits a unique extension to an element of $D^r(G, E)$ if (and only if) $r < k + 1$. This reduces to the theorem above for $k = 0$.
   \end{remark}

  \subsection{Groups with a quasi-factorization}
   \label{sect:quasifact}

   The constructions we have used for $C^r$ and $D^r$ seem to be essentially the only sensible approaches to defining such spaces which are functorial with respect to arbitrary automorphisms of $G$. However, if one assumes a little more structure on the group $G$ then there are other possibilities, which allow a greater range of locally constant distributions to be $p$-adically interpolated.

   \begin{definition}
    Let $G$ be an abelian $p$-adic analytic group, and let $\cG$ be the Lie algebra of $G$. A \emph{quasi-factorization} of $G$ is a decomposition $\cG = \bigoplus \cH_i$ of $\cG$ as a direct sum of subspaces.
   \end{definition}

   If $(\cH_i)$ is a quasi-factorization of $G$, we can clearly choose (non-uniquely) closed subgroups $H_i \subseteq G$ such that $\cH_i$ is the Lie algebra of $H_i$ and $\prod H_i$ is an open subgroup of $G$. We shall say that $(H_i)$ are subgroups compatible with the quasi-factorization $(\cH_i)$.

   \begin{definition}
    Let $G$ be an abelian $p$-adic analytic group with a quasi-fact\-orization $\cH = (\cH_1, \dots, \cH_g)$, and let $(H_1, \dots, H_g)$ be subgroups compatible with $\cH$. Let $f: G \to E$ be continuous. Define
    \[ \varepsilon_{m_1, \dots, m_g}(f) = \inf_{\substack{x \in G \\ y \in p^{m_1} H_1 \times \dots \times p^{m_g} H_g}} v_p\left[ f(x + y) - f(x) \right].\]

    We say $f$ has \emph{order $(r_1, \dots, r_g)$}, where $r_i \in [0, 1)$, if we have
    \begin{equation}\label{eq:qfactbound}
     \varepsilon_{m_1, \dots, m_g}(f) - (r_1 m_1 + \dots + r_g m_g) \to \infty
    \end{equation}
    as $(m_1, \dots, m_g) \to \infty$ (with respect to the filter of cofinite subsets of $\mathbf{N}^g$, i.e.~for any $N$ there are finitely many $g$-tuples $(m_1, \dots, m_g)$ such that the above expression is $\le N$)
   \end{definition}

   We define $C^{(r_1, \dots, r_g)}(G, E)$ to be the space of functions of order $(r_1, \dots, r_g)$, equip\-ped with the valuation given by
   \[ v_{(r_1, \dots, r_g)}(f) = \inf\left[\inf_{x \in G} v_p(f(x)), \inf_{(m_1, \dots, m_n)}\left( \varepsilon_{m_1, \dots, m_g}(f) - (r_1 m_1 + \dots + r_g m_g) \right)\right],\]
   which makes it into an $E$-Banach space. We write $D^{(r_1, \dots, r_g)}(G, E)$ for its dual. It is clear that as topological vector spaces these do not depend on the choice of the $(H_i)$.

   \begin{remark}
    One can check that if $G = H_1 \times \dots \times H_g$, then $C^{(r_1, \dots, r_g)}(G, E)$ is isomorphic to the completed tensor product
    \[ \mathop{\widehat\bigotimes}_{1 \le i \le g} C^{r_i}(H_i, E).\]
   \end{remark}

   The spaces $C^{(r_1, \dots, r_g)}(G, E)$ are clearly invariant under automorphisms of $G$ preserving the quasi-factori\-zation $\cH$.

   \begin{definition}
    \label{def:quasifactgrowth}
    Let $G$ be an abelian $p$-adic analytic group with a quasi-factor\-ization $\cH = (\cH_1, \dots, \cH_g)$, and let $(H_1, \dots, H_g)$ be subgroups compatible with $\cH$. Let $\mu$ be a locally constant distribution on $G$.

    We say $\mu$ has \emph{growth bounded by $(r_1, \dots, r_d)$} if the following condition holds: there is a constant $C$ such that for all $m_1, \dots, m_g \in \ZZ_{\ge 0}$, we have
    \[ \inf_{a \in G} v_p \mu\left(1_{a + (p^{m_1} H_1) \times \dots \times (p^{m_g} H_g)}\right) \ge C - (r_1 m_1 + \dots + r_g m_g).\]
   \end{definition}

   It is clear that whether or not $\mu$ has growth bounded by $(r_i)$ does not depend on the choice of $(H_i)$, although the constant $C$ will so depend.

   \begin{theorem}
    \label{thm:quasifactthm}
    Suppose $\mu$ is a locally constant distribution with growth bounded by $(r_1, \dots, r_g)$, where $r_i \in [0, 1)$. Then there is a unique extension of $\mu$ to an element \[ \tilde\mu \in D^{(r_1, \dots, r_g)}(G, E).\]
   \end{theorem}

   \begin{proof}
    As usual it suffices to assume that $G = H_1 \times \dots \times H_g$. Then the isomorphism
    \[ C^{(r_1, \dots, r_g)}(G, E) \cong \mathop{\widehat\bigotimes}_{1 \le i \le g} C^{r_i}(H_i, E)\]
    gives an explicit Banach space basis of $C^{(r_1, \dots, r_g)}(G, E)$, and the result is now clear.
   \end{proof}

   \begin{remark}
    The special case of Theorem \ref{thm:quasifactthm} where $G = \Zp^2$, with its natural quasi-factorization, is studied in \S 7.1 of \cite{kim-preprint}, which was the inspiration for many of the results of this paper.
   \end{remark}

 \section{Distributions on ray class groups}

  We now consider the application of the above machinery to global settings. We give a fairly general formulation, in the hope that this theory will be useful in contexts other than those we consider in the sections below; sadly, this comes at the cost of somewhat cumbersome notation.

  Let $K$ be a number field. As usual, we define a ``modulus'' of $K$ to be a finite formal product $\prod_{i = 1}^N v_i^{n_i}$ where each $v_i$ is either a finite prime of $K$ or a real place of $K$, and $n_i \in \ZZ_{\ge 0}$, with $n_i \le 1$ if $v_i$ is infinite. We define a ``pseudo-modulus'' to be a similar formal product but with some of the $n_i$ at finite places allowed to be $\infty$. If $\fF$ is a pseudo-modulus, then the ray class group of $K$ modulo $\fF$ is defined, and we denote this group by $G_{\fF}$.

  Let $\fF$ be a pseudo-modulus of $K$, and let $\Sigma$ be the finite set of primes $\fp$ of $K$ such that $\fp^\infty \mid \fF$.

  Let $F$ be a finite extension of either $\QQ$ or of $\Qp$ for some prime $p$, and $V$ a finite-dimensional $F$-vector space.

  \begin{definition}
   A \emph{growth parameter} is a family $\delta = (\delta_{\fp})_{\fp \in \Sigma}$ of non-zero elements of $\cO_F$ indexed by the primes in $\Sigma$.
  \end{definition}

  \begin{definition}
   A \emph{ray class distribution} modulo $\fF$ with growth parameter $\delta$ and values in $V$ to be the data of, for each modulus $\ff$ dividing $\fF$, an element
   \[ \Phi_{\ff} \in F[G_{\ff}] \otimes_{F} V,\]
   such that:
   \begin{enumerate}
    \item \label{cond1} We have $\Norm_{\ff}^{\ff'} \Phi_{\ff'} = \Phi_{\ff}$ for all pairs of moduli $(\ff, \ff')$ with $\ff \mid \ff'$ and $\ff' \mid \fF$.
    \item \label{cond2} There exists an $\cO_F$-lattice $\Lambda \subseteq V$ such that for every $\ff \mid \fF$, we have $\Phi_{\ff} \in \cO_F[G_{\ff}] \otimes_{\cO_F} \delta_{\ff}^{-1} \Lambda$,
    where $\delta_{\ff} = \prod_{\fp \in \Sigma} \delta_{\fp}^{-v_\fp(\ff)}$.
   \end{enumerate}
  \end{definition}

  We shall see in subsequent sections that systems of elements of this kind appear in several contexts as steps in the construction of $p$-adic $L$-functions attached to automorphic forms.

  Since the natural maps $G_{\fF} \to G_{\ff}$ for $\ff \mid \fF$ are surjective, and $G_{\fF}$ is equal to the inverse limit of the finite groups $G_{\ff}$ over moduli $\ff \mid \fF$, condition \eqref{cond1} implies that we may interpret a ray class distribution as a locally constant distribution on $G_{\fF}$ with values in $V$. Condition \eqref{cond2} can then be interpreted as a growth condition on this locally constant distribution. Our goal is to investigate how this interacts with the growth conditions studied in \S \ref{sect:padicanalysis} above.

  If the coefficient field $F$ is a number field, then choosing a prime $\fq$ of $F$ allows us to interpret $V$-valued ray class distributions as $V_{\fq}$-valued ray class distributions, where $V_{\fq} = F_{\fq} \otimes_{F}V$. So we shall assume henceforth that $F = E$ is a finite extension of $\Qp$ for some rational prime $p$, as in \S \ref{sect:padicanalysis}.

  Let us write $\fF = \fg \cdot \prod_{\fp \in \Sigma} \fp^\infty$, for some modulus $\fg$ coprime to $\Sigma$. Then we have an exact sequence
  \[ 0 \to \overline{\cE}_\fg \to \prod_{\fp \in \Sigma} \cO_{K, \fp}^\times \to G_{\fF} \to G_{\fg} \to 0\]
  where $\overline{\cE}_\fg$ is the closure of the image in $\prod_{\fp \in \Sigma} \cO_{K, \fp}^\times$ of the group of units of $\cO_K$ congruent\footnote{We adopt the usual convention in class field theory that if $\fg$ is divisible by a real infinite place $v$, ``congruent to 1 modulo $v$'' means that the image of the unit concerned under the corresponding real embedding of $K$ should be positive.} to 1 modulo $\fg$. Since $G_\fg$ is finite, we see that if $\Sigma$ is a subset of the primes above $p$, then $G_{\fF}$ is a $p$-adic analytic group; and if
  \[ \ff = \prod_{\fp \in \Sigma} \fp^{n_{\fp}}\]
  for some integers $n_{\fp}$, the kernel of the natural surjection $G_{\fF} \rightarrow G_{\ff \fg}$ is the image $H_\ff$ of the subgroup
  \[ U_\ff \coloneqq \prod_{\fp \in \Sigma} (1 + \fp^{n_\fp} \cO_{K, \fp})^\times \subseteq \prod_{\fp \in \Sigma} \cO_{K, \fp}^\times.\]

  The subgroups $H_{\ff}$ for $\ff \mid \Sigma^\infty$ form a basis of neighbourhoods of 1 in $G_{\fF}$, so we may regard a ray class distribution as a locally constant distribution on $G_{\fF}$ satisfying a boundedness condition relative to the subgroups $H_{\ff}$. Our task is to investigate how this interacts with the boundedness conditions considered in \S \ref{sect:padicanalysis} above.

  \subsection{Criteria for distributions of order r}

   The following theorem gives a sufficient condition for a ray class distribution to define a finite-order distribution on the group $G_{\fF}$ in the sense of \S \ref{sect:severalvars1}, when $\Sigma$ is a subset of the of primes above $p$ (so that $G_{\fF}$ is a $p$-adic analytic group). We continue to assume that the coefficient field $E$ is a finite extension of $\Qp$, and (as before) we write $v_p$ for the valuation on $E$ such that $v_p(p) = 1$. We shall assume for simplicity that the coefficient space $V$ is simply $E$; the general case follows immediately from this.

   \begin{theorem}
    \label{thm:rayclassdist1param}
    Let $\Theta$ be a ray class distribution modulo $\fF$. Suppose $\Theta$ has growth parameter $(\alpha_\fp)_{\fp \in \Sigma}$, and define
    \[ r = \sum_{\fp \in \Sigma} e_\fp v_p(\alpha_\fp),\]
    where $e_{\fp}$ is the absolute ramification index of the prime $\fp$. Then $\Theta_\fF$, viewed as a locally constant distribution on $G_{\fF}$, has growth bounded by $r$ in the sense of Definition \ref{def:severalvarsgrowth}.

    In particular, if $r < 1$, there is a unique element $\tilde \Theta \in D^r(G_{\fF}, E)$ extending $\Theta$.
   \end{theorem}

   \begin{proof}
    Let $\fp\in \Sigma$. We can choose some integer $c \ge 0$ such that the $p$-adic logarithm converges on $U_{\fp, k} = (1 + \fp^{m} \cO_{\fp})^\times \subseteq \cO_{\fp}^\times$ for all $k \ge c$, and identifies $U_{\fp, k}$ with the additive group $\fp^{k} \cO_{\fp}$. In particular, $U_{\fp, c}$ is isomorphic to $\Zp^{[K_{\fp} : \Qp]}$, and for $m \ge 0$ we have
    \[ U_{\fp, c}^{p^m} \cong p^m \fp_i^{c} \cO_{\fp} = \fp^{c + e_\fp m} \cO_{\fp} \cong U_{\fp, c + e_\fp m}.\]

    Now let $\fp_1, \dots, \fp_g$ be the primes in $\Sigma$ and choose a constant $c_i$ for each. Then the subgroup $U = U_{\fp_1, c_1} \times \dots \times U_{\fp_g, c_g}$ is an open subgroup of $(\cO_K \otimes \Zp)^\times = \prod_{i = 1}^g \cO_{\fp_i}^\times$, and $U \cong \Zp^{[K : \QQ]}$. By increasing the $c_i$ if necessary, we may assume that the image of $U$ in $G_{\fF}$ is torsion-free and hence isomorphic to $\Zp^{h}$ for some $h \le [K : \QQ]$.

    Let $H$ be the image of $U$. Then the image of $U^{p^m}$ is $H^{p^m}$, clearly; so a locally constant distribution on $G_\fF$ has order $r$ if and only if there is $C$ such that $v_p \mu( 1_{x \cdot H^{p^m}}) \ge C - rm$ (this is just Definition \ref{def:severalvarsgrowth}, but with the group law on $G_{\fF}$ written multiplicatively).

    However, we have
    \[ U^{p^m} = \prod_{i = 1}^g U_{\fp_i, c_i + e_{\fp_i} m} = \prod_{i = 1}^g (1 + \fp_i^{c_i + e_{\fp_i} m} \cO_{\fp_i})^\times.\]
    Thus $G_{\fF} / H^{p^m}$ is the ray class group of modulus
    \[ \ff_m \coloneqq \fg \cdot \prod_{i = 1}^g \fp_i^{c_i + e_{\fp_i} m},\]
    where $\fg$ is (as above) the modulus such that $\fF = \fg \cdot \prod_{i = 1}^g \fp_i^\infty$.
    So, if $\Theta$ is a ray class distribution with growth parameter $(\alpha_\fp)_{\fp \in \Sigma}$, then we know that the valuation of $\Theta(X)$ where $X$ is any coset of $H^{p^m}$ is bounded below by
    \[ C - \sum_{i = 1}^g v_p(\alpha_{\fp_i}) \cdot \ord_{\fp_i}(\ff_m) = C' - m \sum_{i = 1}^g e_{\fp_i} v_p(\alpha_{\fp_i}) = C' - r m\]
    for some constants $C, C'$, where $r$ is as defined in the statement of the theorem. So a ray class distribution with growth parameter $(\alpha_\fp)_{\fp \in \Sigma}$ defines a locally constant distribution on the $p$-adic analytic group $G_{\fF}$ whose growth is bounded by $r$, as required.
   \end{proof}

   Note that this argument does not depend on the dimension of $G_{\fF}$ (which is useful, since this dimension depends on whether Leopoldt's conjecture holds for $K$). This result is essentially the best possible (at least using the present methods) when there is a unique prime of $K$ above $p$, or when $K$ is totally real and Leopoldt's conjecture holds (as we show in the next section). However, for other fields $K$ finer statements are possible using the theory of quasi-factorizations developed in \S \ref{sect:quasifact}, as we shall see below.

  \subsection{A converse result for totally real fields}

   Let us now suppose $K$ is totally real. We also suppose that Leopoldt's conjecture holds for $K$, so the image of $\cE$ in $\prod_{\fp \mid p} \cO_\fp^\times$ has rank $n-1$. We shall take $\fF = p^\infty \fg$ for some modulus $\fg$ coprime to $p$, so $\Sigma = \Sigma_p$ is the set of all primes dividing $p$. Leopoldt's conjecture implies that $G_{p^\infty \fg}$ is a $p$-adic analytic group of dimension 1; thus, in particular, the notions of finite-order functions and distributions on $G_{p^\infty \fg}$ are just the standard ones.

   We shall prove the following theorem:

   \begin{theorem}
    \label{thm:converse}
    Let $\underline\alpha = (\alpha_{\fp})_{\fp \in \Sigma_p}$ be elements of $E$, and let $h \in \RR_{\ge 0}$. Then the implication
    \begin{quotation}
     ``Every ray class distribution modulo $p^\infty \fg$ of growth parameter $\underline\alpha$ is a locally constant distribution on $G_{p^\infty \fg}$ with growth bounded by $h$''
    \end{quotation}
    is true if and only if the inequality
    \[ \sum_{\fp \in \Sigma_p} e_{\fp} v_p(\alpha_{\fp}) \le h\]
    holds.
   \end{theorem}

   We retain the notation of the proof of Theorem \ref{thm:rayclassdist1param}, so $\fp_1, \dots, \fp_g$ are the primes above $p$, and for each $i$, $c_i$ is a constant such that for all $k \ge c_i$ the logarithm map identifies $U_{\fp_i, k}$ with the additive group $(\fp_i)^m$, and the image of $U = U_{\fp_1, c_1} \times \dots U_{\fp_g, c_g}$ in $G$ is torsion-free and hence isomorphic to $\Zp$.

   Enlarging the $c_i$ further if necessary, we may assume that there is an integer $w$ independent of $i$ such that for each $i$ the image of $\fp_i^{c_i}$ under the trace map $Tr_{K_{\fp_i} / \Qp}$ is $p^w \Zp$. Let $d_i = [K_{\fp_i} : \Qp]$.

   \begin{proposition}
    For each $i$ there exists a basis $b_{i1}, \dots, b_{i d_i}$ of $\fp_i^{c_i}$ as a $\Zp$-module such that
    \[ \Tr_{K_{\fp_i} / \Qp}(b_{ij}) = p^w\]
    for each $j$.
   \end{proposition}

   \begin{proof}
    Elementary linear algebra.
   \end{proof}

   \begin{proposition}
    The isomorphism
    \[ U \cong \Zp^{n}\]
    (where $n = [K : \QQ] = \sum_{i=1}^g d_i$) given by the bases $b_{i1}, \dots, b_{i d_i}$ for $1 \le i \le g$ identifies the closure of $\cE \cap U$ with the submodule
    \[ \Delta = \{(x_1, \dots, x_d) \in \Zp^d : \sum_{j=1}^d x_j = 0\}.\]
   \end{proposition}

   \begin{proof}
    With respect to the basis $b_1, \dots, b_{e_i f_i}$ of $\fp_i^{c_i}$, the trace map is given by $(x_1, \dots, x_{d_i}) \mapsto p^w \sum_j x_j$, so our isomorphism
    \[ U \stackrel{\log}{\longrightarrow} \prod_{i = 1}^g \fp_i^{c_i} \stackrel{\cong}{\longrightarrow} \Zp^n\]
    identifies $\Delta$ with the elements of $U$ whose norm down to $\QQ$ is 1. However, since $U$ is torsion-free, any $u \in U \cap \cE$ satisfies $\Norm_{K / \QQ}(u) = 1$, so $\Delta$ contains the closure of $\cE \cap U$, which we write as $\Delta'$.

    By Leopoldt's conjecture (which we are assuming), $\Delta'$ has $\Zp$-rank $(n-1)$, the same as $\Delta$; so the quotient $\frac{\Delta}{\Delta'}$ is finite. However, we obviously have
    \[ \frac{U}{\Delta'} \cong \Zp \oplus \frac{\Delta}{\Delta'};\]
    and $\frac{U}{\Delta'} = U(p^\infty, \fp_1^{c_1} \dots \fp_g^{c_g})$ is torsion-free, so we must have $\Delta = \Delta'$.
   \end{proof}

   \begin{proposition}
    Let $H$ be the image of $U$ in $G_{p^\infty \fg}$, and let $H_m = H^{p^m}$ as above. Suppose $m \ge 1$. Then if $r_1, \dots, r_g$ are integers such that the image in $G$ of $U_{\fp_1, r_1} \times \dots \times U_{\fp_g, r_g}$ is contained in $H_m$, we must have $r_i \ge c_i + e_i m$ for all $i$.
   \end{proposition}

   \begin{proof}
    If the image of $U_{\fp_1, r_1} \times \dots \times U_{\fp_g, r_g}$ is contained in $H_m$, then the same must also be true with $(r_1, \dots, r_g)$ replaced by $(r_1', \dots, r_g')$ where $r_i' = \sup(r_i, c_i)$. So we may assume without loss of generality that $r_i \ge c_i$ for all $i$. Now the result is clear from the above description of the image of the global units.
   \end{proof}

   This result clearly implies Theorem \ref{thm:converse}.

   \begin{remark}
    Curiously, the results above seem to conflict with some statements asserted without proof in \cite{panchishkin94} (and quoted by some other subsequent works). The group studied in \emph{op.cit.} corresponds in our notation to $G_{\fg p^\infty}$, where $\fg$ is the product of the infinite places, and $K$ is assumed totally real. (Thus $G_{\fg p^\infty}$ corresponds via class field theory to the Galois group of the maximal abelian extension of $K$ unramified outside $p$ and the infinite places; it is denoted by $\operatorname{Gal}_p$ in \emph{op.cit.}.)

    In Definition 4.2 of \emph{op.cit.}, a locally constant distribution $\mu$ on $G_{\fg p^\infty}$ is defined to be \emph{$1$-admissible} if
    \[ \sup_a \left|\mu\left(1_{a + (\fm)}\right)\right| = o\left(|\fm|_p^{-1}\right).\]
    Paragraph 4.3 of \emph{op.cit.} then claims that a 1-admissible measure extends uniquely to an element of $D^{\la}(G_{\fg p^\infty}, E)$ which, regarded as a rigid-analytic function on the space $\mathcal{X}_p$ parametrizing characters of $G_{\fg p^\infty}$, has growth $o(\log(1 + X))$.

    This is consistent with Theorem \ref{thm:rayclassdist1param} and Theorem \ref{thm:converse} above if there is only one prime above $p$. However, if there are multiple primes above $p$, it is not so clear how $|\fm|$ is to be defined for general ideals $\fm \mid p^\infty$ of $K$, and the uniqueness assertion of Conjecture 6.2 of \emph{op.cit.} (which is asserted to be a consequence of the theory of $h$-admissible measures) contradicts Theorem \ref{thm:converse} of the present paper.
   \end{remark}

  \subsection{Imaginary quadratic fields}

   Let $K$ be an imaginary quadratic field. If $p$ is inert or ramified in $K$ then there is only one prime $\fp$ of $K$ above $p$, and hence there is no canonical direct sum decomposition of $\Lie G_{p^\infty} \cong K_{\fp}$. On the other hand, we can find one when $p$ splits:

   \begin{proposition}
    If $K$ is imaginary quadratic and $p = \fp \fpb$ is split, then the images of $\Lie \cO_{K, \fp}^\times$ and $\Lie \cO_{K, \fpb}^\times$ in $\Lie(G_{p^\infty})$ form a quasi-factorization.
   \end{proposition}

   \begin{proposition}
    Let $\mu$ be a ray class distribution modulo $p^\infty \fg$ with growth parameter $(\alpha_{\fp}, \alpha_{\fpb})$. Then $\mu$ has growth bounded by $(v_p(\alpha_\fp), v_p(\alpha_{\fpb}))$ in the sense of Definition \ref{def:quasifactgrowth}.
   \end{proposition}

   \begin{proof}
    Clear by construction.
   \end{proof}

   Combining the above with Theorem \ref{thm:quasifactthm} we have the following:

   \begin{theorem}
    Let $K$ be an imaginary quadratic field in which $p = \fp \fpb$ is split, and let $\Theta$ be a ray class distribution modulo $p^\infty \fg$ with values in $V$ and growth parameter $(\alpha_{\fp}, \alpha_{\fpb})$. Let $r = v_p(\alpha_\fp)$ and $s = v_p(\alpha_{\fpb})$. If $r, s < 1$, then there is a unique distribution
    \[ \widetilde\Theta \in D^{(r, s)}(G_{p^\infty}, E) \otimes_E V \]
    such that the restriction of $\widetilde\Theta$ to $LC(G_{p^\infty}, E) \otimes_E V$ is $\Theta$. If we have $r + s < 1$, then $\widetilde\Theta$ lies in $D^{r + s}(G_{p^\infty}, E) \otimes_{E} V$ and agrees with the element constructed in Theorem \ref{thm:rayclassdist1param}.
   \end{theorem}

  \subsection{Fields of higher degree}
   \label{sect:higherdegree}

   When $K$ is a non-totally-real extension of degree $>2$, one can sometimes find interesting quasi-factorizations of $G_{p^\infty}$ by considering subsets of the primes above $p$. If we identify $G_{p^\infty}$ with the Galois group of the maximal abelian extension $K(p^\infty) / K$ unramified outside $p$, we seek to write $K(p^\infty)$ as a compositum of smaller extensions, each ramified at as few primes as possible. We give a few examples of the behaviour that occurs when $p$ is totally split.

   \begin{example}[Mixed signature cubic fields]
    Suppose $K$ is a non-totally-real cubic field, and $p$ is totally split in $K$, say $p = \fp_1 \fp_2 \fp_3$. Then $G_{p^\infty}$ has dimension 2 (since Leopoldt's conjecture is trivially true in this case), and the images of the groups $\cO_{K, \fp_i}^\times$ in $G_{p^\infty}$ give three pairwise-disjoint one-dimensional subspaces of $\Lie G_{p^\infty}$, so any two of these form a quasi-factorization. One therefore obtains three quasi-factorizations of $G_{p^\infty}$.

    One checks that with respect to the quasi-factorization given by $\fp_1$ and $\fp_2$, a ray class distribution of parameter $(\alpha_{\fp_1}, \alpha_{\fp_2}, \alpha_{\fp_3})$ has growth bounded by $(r,s)$ if and only if $v_p(\alpha_{\fp_1}) + v_p(\alpha_{\fp_3}) \le r$ and $v_p(\alpha_{\fp_2}) + v_p(\alpha_{\fp_3}) \le s$. For instance, if all $\alpha_i$ are equal and their common valuation $h$ is $ < \tfrac12$, then can construct a locally analytic distribution (indeed three of them), while Theorem \ref{thm:rayclassdist1param} would only apply if $h < \tfrac13$. For the cases $\tfrac13 \le h < \tfrac12$, it is not immediately obvious whether the interpolations provided by the three quasi-factorizations are equal or not, but we shall see in the next section that this is indeed the case.
   \end{example}

   \begin{example}[Quartic fields]
    Suppose $K$ is a totally imaginary quartic field, and $p$ is totally split in $K$. As in the previous example, $\cO_K^\times$ has rank 1 and thus Leopoldt's conjecture is automatic, so $G_{p^\infty}$ has rank 3. We can obtain a quasi-factorization of $G_{p^\infty}$ by considering the images of the Lie algebras of $\cO_{K, \fp_i}^\times$ for $1 \le i \le 3$; then (much as in the previous case) we find that the condition for a ray class distribution to have growth bounded by $(r, s, t)$ is that
    \begin{align*}
     v_p(\alpha_1) + v_p(\alpha_4) &\le r,\\
     v_p(\alpha_2) + v_p(\alpha_4) &\le s,\\
     v_p(\alpha_3) + v_p(\alpha_4) &\le t.\\
    \end{align*}
    A slightly more intricate class of quasi-factorizations can be obtained as follows. We can choose a basis for $\Lie \cO_{K, \fp_i}^\times$ for each $i$ such that the Lie algebra of the subgroup $\overline{\cE}_K$ corresponds to the subspace $\Delta$ of $\Qp^4$ spanned by $(1,1,1,1)$. Then choosing a quasi-factorization amounts to choosing a basis of the 3-dimensional space of linear functionals on $\Qp^4$ which vanish on $\Delta$. One such basis is given by the functionals mapping $(x_1, x_2, x_3, x_4)$ to $\{ x_1 - x_4, x_2 - x_4, x_3 - x_4\}$, and this gives the quasi-factorization we have already seen. However, we can also consider the functionals $\{ x_1 - x_2, x_2 - x_3, x_3 - x_4\}$. One checks then that the condition for a ray class distribution to have growth bounded by $(r, s, t)$ is
    \begin{align*}
     v_p(\alpha_1) + v_p(\alpha_2) &\le r,\\
     v_p(\alpha_2) + v_p(\alpha_3) &\le s,\\
     v_p(\alpha_3) + v_p(\alpha_4) &\le t.\\
    \end{align*}
   \end{example}

   One can generalize the last example to a much wider range of number fields under suitable assumptions on the position of $\Lie \overline\cE_K$ inside $\Lie (\Zp \otimes \cO_K)^\times$ (which amounts to assuming a strong form of Leopoldt's conjecture, such as that formulated in \cite{calegarimazur09}).

  \subsection{A representation-theoretic perspective} For simplicity let us suppose that $K$ has class number 1 and $\fg = (1)$, so $G :=G_{p^\infty} = U / \overline{\Delta}$ where $U = \prod_{i = 1}^g \cO_{\fp_i}^\times$ and $\Delta = \cO_K^\times$. Pick real numbers $r_1, \dots, r_g$ with $r_i \in [0, 1) \cap v_p(E^\times)$.

   We consider the inclusion
   \[ LC(U, E)^\Delta \into C^{(r_1, \dots, r_g)}(U, E)^{\Delta},\]
   where we give $U$ the obvious quasi-factorization. We have, of course, $LC(U, E)^\Delta = LC(U / \overline{\Delta}, E) = LC(G, E)$. The following proposition is immediate from the definitions:

   \begin{proposition}
    Let $\mu$ be an element of $D^{(r_1, \dots, r_g)}(U, E)$. Then the restriction of $\mu$ to $LC(U, E)^\Delta$ is a ray class distribution of parameter $(\alpha_1, \dots, \alpha_g)$, where the $\alpha_i$ are any elements with $v_p(\alpha_i) = r_i$.

    Conversely, any ray class distribution of parameter $(\alpha_1, \dots, \alpha_g)$ defines a linear functional on $LC(U, E)^\Delta$ which is continuous with respect to the subspace topology given by the inclusion into $C^{(r_1, \dots, r_g)}(U, E)$.
   \end{proposition}

   However, even though $LC(U, E)$ is dense in $C^{(r_1, \dots, r_g)}(U, E)$, it does not necessarily follow that  $LC(U, E)^\Delta$ is dense in $C^{(r_1, \dots, r_g)}(U, E)^{\Delta}$; the completion of the invariants may be smaller than the invariants of the completion.

   \begin{example}
     Consider the case where $U$ is $\Zp^2$ with its natural quasi-factorization, and $\overline{\Delta}$ is the subgroup $\{(x, y) : x + y = 0\}$. Then a continuous function $f: U \to E$ is $\Delta$-invariant if and only if there is a function $h \in C^0(\Zp, E)$ such that $f(x, y) = h(x + y)$. So the coefficients in the Mahler--Amice expansion $f(x, y) = \sum_{m, n} a_{m,n} \binom x m \binom y n$ are given by $a_{m, n} = b_{m + n}$, where $h(x) = \sum_{n \ge 0} b_n \binom x n$. Since the functions \[ p^{\lfloor r_1 \ell(m) + r_2 \ell(n)\rfloor} \binom x m \binom y n\]
     are a Banach basis of $C^{(r_1, r_2)}(U, E)$, we see that $f$ is in this space if and only if
     \[ \lim_{m, n \to \infty} v_p(b_{m + n}) - r_1 \ell(m) - r_2 \ell(n) = \infty.\]

     The supremum of $r_1 \ell(m) + r_2 \ell(n)$ over pairs $(m, n)$ with $m + n = k$ is $(r_1 + r_2)\ell(k) + O(1)$ as $k \to \infty$, so this is equivalent to
     \[ \lim_{k \to \infty} v_p(b_k) - (r_1 + r_2) \ell(k) = \infty,\]
     which is precisely the condition that $h$ is $C^{r}$ where $r = r_1 + r_2$. Moreover, this gives a Banach space isomorphism between $C^r(\Zp, E)$ and $C^{(r_1, r_2)}(U, E)^\Delta$ preserving the subspaces of locally constant functions, so if $r_1 + r_2 \ge 1$ we see that $LC(U, E)^\Delta$ is not dense in $C^{(r_1, r_2)}(U, E)^\Delta$.
    \end{example}

    \begin{remark}
     In \S \ref{sect:higherdegree}, we gave criteria for locally constant functions on $U / \overline\Delta$ to be dense in a Banach space whose topology is coarser than that of $C^{(r_1, \dots, r_g)}(U, E)^\Delta$. This in particular implies that such functions are dense in $C^{(r_1, \dots, r_g)}(U, E)^\Delta$ in these cases, and hence that the multiple possible choices of auxilliary quasi-factorizations considered in the examples above do all give the same distribution on $G$ when they apply.
    \end{remark}

 \section{Construction of $p$-adic $L$-functions}

  We now give the motivating example of a ray class distribution: the distribution constructed from the Mazur--Tate elements of an automorphic representation of $\GL_2 / K$.

  \subsection{Mazur--Tate elements}
   \label{sect:mtdefs}

   Let $K$ be a number field, and $\Pi$ an automorphic representation of $\GL_2 / K$. We make the following simplifying assumptions:

   \begin{enumerate}
    \item $\Pi$ is cohomological in trivial weight, i.e.~the $(\fg, K_\infty)$-cohomology of $\Pi$ does not vanish, where $K_\infty$ is a maximal compact connected subgroup of $\GL_2(K \otimes \RR)$;
    \item the central character of $\Pi$ is trivial.
   \end{enumerate}

   Conditions (1) and (2) are satisfied, for instance, if $\Pi$ is the base-change to $K$ of the automorphic representation of $\GL_2 / \QQ$ attached to a modular form of weight 2 and trivial nebentypus. It follows from (1) and (2) that there exists a finite extension $F / \QQ$ inside $\CC$ such that the $\GL_2(\AA_K^\infty)$-representation $\Pi^{\infty}$ can be defined over $F$, and in particular the Hecke eigenvalues $a_\fl(\Pi)$ for each prime $\fl$ of $K$ are in $\cO_F$.

   \begin{theorem}[Haran]
    \label{thm:haran}
    Under the above assumptions, there exists a finitely-generated $\cO_F$-sub\-module $\Lambda_\Pi$ of $\CC$, and for each ideal $\ff$ of $K$ coprime to the conductor of $\Pi$ an element
    \[ \Theta_\ff(\Pi) \in \ZZ[G_\ff] \otimes_{\ZZ} \Lambda_\Pi,\]
    where $G_{\ff}$ is the ray class group modulo $\ff$, such that the following relations hold:
    \begin{enumerate}[(i)]
     \item (Special values) If $\omega$ is a primitive ray class character of conductor $\ff$, then
     \[ \omega\left(\Theta_\ff(\Pi)\right) = \frac{L_{\ff}(\Pi, \omega, 1)}{\tau(\omega) \cdot |\ff|^{1/2} \cdot (4\pi)^{[K : \QQ]}}\]
     where $L_\ff(\Pi, \omega, s)$ denotes the $L$-function of $\Pi$ twisted by $\omega$, without the Euler factors at $\infty$ or at primes dividing $\ff$, and $\tau(\omega)$ is the Gauss sum (normalized so that $|\tau(\omega)| = 1$);
     \item (Norm relation) For $\fl$ a prime not dividing the conductor of $\Pi$, we have
     \[ \Norm_{\ff}^{\ff \fl} \Theta_{\ff \fl}(\Pi) =
      \begin{cases}
       \left(a_\fl(\Pi) - \sigma_{\fl} - \sigma_{\fl}^{-1}\right) \Theta_{\ff}(\Pi) & \text{if $\fl \nmid \ff$,}\\
       a_\fl(\Pi) \Theta_{\ff}(\Pi) - \Theta_{\ff / \fl}(\Pi) & \text{if $\fl \mid \ff$.}
      \end{cases}
     \]
     where $\sigma_{\fl}$ denotes the class of $\fl$ in $G_{\ff}$.
    \end{enumerate}
   \end{theorem}

   \begin{proof}
    See \cite{haran87}. (Note that in the main text of \emph{op.cit.} it is assumed that $K$ is totally imaginary, but this assumption can easily be removed, as noted in section 7.)
   \end{proof}

   \begin{remark}
    In fact one can obtain essentially the same result under far weaker assumptions: it suffices to suppose that $\Pi$ is cuspidal, cohomological (of any weight), and critical. However, the necessary generalizations of Shai Haran's arguments are time-consuming, so we shall not consider this more general setting here. For the case when $K$ is totally real, but $\Pi$ has arbitrary critical weight and central character, see \cite{dimitrov13}.
   \end{remark}

   We refer to the elements $\Theta_\ff(\Pi) \in \ZZ[G_\ff] \otimes_{\ZZ} \Lambda_\Pi$ as \emph{Mazur--Tate elements}, since they are closely analogous to the group ring elements considered in \cite{mazurtate87}. The coefficient module $\Lambda_{\Pi}$ is in fact rather small, as the following result (an analogue of the Manin--Drinfeld theorem) shows:

   \begin{proposition}
    \label{prop:manindrinfeld}
    The $\cO_F$-submodule $\Lambda_{\Pi} \subseteq \CC$ may be taken to have rank 1.
   \end{proposition}

   \begin{proof}
    We know that $\Lambda_{\Pi}$ is finitely-generated, so it suffices to show that $F \otimes_{\cO_F} \Lambda_{\Pi}$ is 1-dimensional.

    To prove this, we must delve a little into the details of Shai Haran's construction. The elements $\Theta_\ff(\Pi)$ for varying $\Pi$ are all obtained from a ``universal Mazur--Tate element'' in the module $\ZZ[G_{\ff}] \otimes H_r^{\mathrm{BM}}(Y, \ZZ)$, where $Y$ is a locally symmetric space for $\GL_2(\AA_K)$ and $H_r^{\mathrm{BM}}$ is Borel--Moore homology (the homological analogue of compactly supported cohomology). Here $r = r_1 + r_2$, where $r_1$ and $r_2$ are the numbers of real and complex places of $K$. The element $\Theta_\ff(\Pi)$ is obtained by integrating the universal Mazur--Tate element against a class in $H^r_{\mathrm{par}}(Y, \CC)$ arising from $\Pi$. However, it follows from the results of \cite{harder87} that there is a unique Hecke-invariant section of the projection map
    \[ H^r_c(Y, \CC) \to H^r_{\mathrm{par}}(Y, \CC), \]
    and the $\Pi$-isotypical component of $H^r_c(Y, \CC)$ is 1-dimensional and descends to $H^r_c(Y, F)$. Thus, after renormalizing by a single (probably transcendental) complex constant depending on $\Pi$, the Mazur--Tate elements for $\Pi$ can be obtained as values of the perfect pairing
    \[ H_r^{\mathrm{BM}}(Y, F) \times H^r_c(Y, F) \to F,\]
    so we deduce that $F \otimes_{\cO_F} \Lambda_{\Pi}$ has dimension 1 as claimed.
   \end{proof}

  \subsection{Stabilization}
   \label{sect:mtpstab}

   We now introduce ``$\Sigma$-stabilized'' versions of the $\Theta_\ff(\Pi)$, again following \cite{haran87} closely. Let $\fF$ be a pseudo-modulus of $K$, and $\Sigma$ the set of primes $\fp$ such that $\fp^\infty \mid \fF$, as before. We assume that none of the primes $\fp \in \Sigma$ divide the conductor of $\Pi$.

   \begin{definition}
    An \emph{$\Sigma$-refinement} of $\Pi$ is the data of, for each $\fp \in \Sigma$, a root $\alpha_\fp \in \overline{F}$ of the polynomial
    \[ P_{\fp}(\Pi) = X^2 - a_{\fp}(\Pi) X + N_{K / \QQ}(\fp) \in F[X]\]
    (the local $L$-factor of $\Pi$ at $\fp$).
   \end{definition}

   We write $\beta_\fp$ for the complementary root to $\alpha_{\fp}$. By enlarging $F$ if necessary, we may assume that the $\alpha_{\fp}$ and $\beta_\fp$ all lie in $F$.

   We introduce a formal operator $R_{\fp}$ on the $\Theta_{\ff}(\Pi)$'s by $R_{\fp} \cdot \Theta_{\ff}(\Pi) = \Theta_{\ff / \fp}(\Pi)$ whenever $\fp \mid \ff$. Then it is easy to see that the $R_{\fp}$ for different $\fp \mid \ff$ commute with each other, so we can define
   \[ \Theta_{\ff}^{\Sigma}(\Pi) = \alpha_{\ff}^{-1} \left( \prod_{\fp \in \Sigma} (1 - \alpha_{\fp}^{-1} R_{\fp}) \right) \Theta_{\ff}(\Pi),\]
   whenever $\ff$ is divisible by all $\fp \in \Sigma$, where $\alpha_\ff$ stands for the product
   \[ \alpha_{\ff} := \prod_{\fp \in \Sigma} \alpha_{\fp}^{-v_{\fp}(\ff)}.\]
   Then one checks that we have
   \[ \Norm_{\ff}^{\ff \fp} \left( \Theta_{\ff\fp}^{\Sigma}(\Pi)\right) = \Theta_{\ff}^{\Sigma}(\Pi)\]
   for any $\fp$ such that $\fp \mid \ff$ and $\fp \in \Sigma$. We can therefore extend the definition of $\Theta_{\ff}^{\Sigma}(\Pi)$ uniquely to all $\ff \mid \fF$ in such a way that this formula still holds.

  \begin{corollary}
   \label{cor:pstab}
   Let $\fF$ be a modulus of $K$ and write $\fF = \fg \fF_\infty$, where $\fF_\infty = \prod_{\fp \in \Sigma} \fp^\infty$ and $\fg$ is coprime to $\Sigma$. Then the elements
   \[ \left\{ \Theta_{\ff\fg}^{\Sigma}(\Pi) : \ff \mid \fF_\infty \right\}
   \]
   defined above form a ray class distribution modulo $\fF$, with values in $F \otimes_{\cO_F} \Lambda_{\Pi}$ and growth parameter $(\alpha_{\fp})_{\fp \in \Sigma}$.
  \end{corollary}

  \subsection{Consequences for p-adic L-functions}

   We now summarize the results on $p$-adic $L$-functions that can be obtained by applying our theory to the ray class distributions constructed in \S \ref{sect:mtpstab}. As before, let $F$ be a number field, and $E$ be the completion of $F$ at some choice of prime above $p$; and let $v_p$ denote the $p$-adic valuation on $E$ normalized so that $v_p(p) = 1$.

   \begin{theorem}
    \label{thm:padicLgeneral}
    Let $K$ be an arbitrary number field and $\Sigma$ a subset of the primes of $K$ above $p$, and let $\fF$ be the pseudo-modulus $\fg \cdot \prod_{\fp \in \Sigma} \fp^\infty$, where $\fg$ is a modulus coprime to $\Sigma$. Let $\Pi$ be an automorphic representation of $\operatorname{GL}_2 / K$ with Hecke eigenvalues in $F$, satisfying the hypotheses (1) and (2) of \S \ref{sect:mtdefs}, and unramified at the primes in $\Sigma$. Let $\underline\alpha = (\alpha_{\fp})_{\fp \in \Sigma}$ be a $\Sigma$-refinement of $\Pi$ defined over $F$.
    If we have
    \[ h = \sum_{\fp \in \Sigma} e_{\fp} v_p(\alpha_{\fp}) < 1,\]
    then there is a unique distribution $L_{p, \underline\alpha}(\Pi) \in D^h(G_{\fF}, E) \otimes_{\cO_F} \Lambda_\Pi$ such that for all finite-order characters $\omega$ of $G_{\fF}$ whose conductor is divisible by $\fg$, we have
    \begin{multline}
     \label{eq:interp}
     L_{p, \underline\alpha}(\Pi)(\omega) =
     \left( \prod_{\fp \in \Sigma} \alpha_{\fp}^{-\ord_\fp \ff}\right) \cdot \left(\prod_{\fp \in \Sigma, \fp \nmid \ff}(1 - \alpha_\fp^{-1} \omega(\fp)) (1 - \alpha_{\fp}^{-1} \omega(\fp)^{-1})\right) \\ \cdot
     \frac{L_{\ff}(\Pi,\omega, 1)}{\tau(\omega) \cdot |\ff| \cdot (4\pi)^{[K : \QQ]}},
    \end{multline}
    where $\ff$ is the conductor of $\omega$.
   \end{theorem}

   \begin{proof}
    This is immediate from Theorem \ref{thm:haran} and Corollary \ref{cor:pstab} except for the formula for $L_{p, \underline\alpha}(\Pi)(\omega)$ when $\ff$ is not divisible by all primes in $\Sigma$; the latter follows via a short explicit calculation from part (ii) of Theorem \ref{thm:haran}.
   \end{proof}

   If $K$ is imaginary quadratic and $p$ is split, we obtain a slightly finer statement using the quasi-factorization above:

   \begin{theorem}
    \label{thm:padicLimquad}
    Let $K$ be an imaginary quadratic field in which $p = \fp \fpb$ splits, and let $\Pi$ be an automorphic representation satisfying the hypotheses (1) and (2) of \S \ref{sect:mtdefs} with Hecke eigenvalues in $F$, and unramified at $\fp$ and $\fpb$. Let $\underline\alpha = (\alpha_{\fp}, \alpha_{\fpb})$ be a $p$-refinement of $\Pi$ defined over $F$.

    If we have
    \[ r = v_p(\alpha_{\fp}) < 1 \quad\text{and}\quad s = v_p(\alpha_{\fpb}) < 1,\]
    then there is a unique distribution $L_{p, \underline\alpha}(\Pi) \in D^{(r, s)}(G_{p^\infty}, E) \otimes_{\cO_F} \Lambda_\Pi$ satisfying the interpolating property \eqref{eq:interp}.
   \end{theorem}

   Note that every $\Pi$ will admit at least one $p$-refinement satisfying the hypotheses of Theorem \ref{thm:padicLimquad}, and indeed if $\Pi$ is non-ordinary at $\fp$ and $\fpb$ then all four $p$-refinements do so; but at most two of the four $p$-refinements, and in many cases none at all, will satisfy the hypotheses of Theorem \ref{thm:padicLgeneral}.

   \begin{remark}
    Applying Theorem \ref{thm:padicLimquad} to the base-change of a classical modular form, we obtain Conjecture 6.13 of our earlier paper \cite{loefflerzerbes11} for all weight 2 modular forms whose central character is trivial.
   \end{remark}

  \begin{remark}
   \label{rmk:othergps}
   Examples of ray class distributions appear to arise in more or less any context where $p$-adic interpolation of $L$-values of non-ordinary automorphic representations over general number fields is considered. For instance, the works of Ash--Ginzburg \cite{ashginzburg94} on $p$-adic L-functions for $\GL_{2n}$ over a number field, and of Januszewski on Rankin--Selberg convolutions on $\GL_n \times \GL_{n-1}$ over a number field \cite{januszewski11}, can both be viewed as giving rise to a ray class distribution (with some parameter depending on the Hecke eigenvalues of the automorphic representations involved). Thus the theory developed in this paper allows these works to be extended to non-ordinary automorphic representations of (sufficiently small) non-zero slope.
  \end{remark}

 \section{The $a_p = 0$ case}

  We now concentrate on a special case of Theorem \ref{thm:padicLimquad}, where $K$ is imaginary quadratic with $p$ split and $a_{\fp}(\Pi) = a_{\fpb}(\Pi) = 0$ (as in the case of the automorphic representation attached to a modular elliptic curve with good supersingular reduction at the primes above $p$, assuming $p \ge 5$). Then the two Hecke polynomials at $\fp$ and $\fpb$ are the same, and we write $\{\alpha, \beta\}$ for their common set of roots, which are both square roots of $-p$ and in particular have normalized $p$-adic valuation $\tfrac{1}{2}$. Thus Theorem \ref{thm:padicLimquad} furnishes four $p$-adic $L$-functions.

  \begin{remark}
   None of these four $L$-functions can be obtained using Theorem \ref{thm:padicLgeneral}. Moreover, the restriction of any of the four to a ``cyclotomic line'' in the space of characters of $G_{p^\infty}$ (i.e.~a coset of the subgroup of characters factoring through the norm map) is a distribution of order 1, and thus not determined uniquely by its interpolating property. Thus, paradoxically, it is easier to interpolate $L$-values at a larger set of twists than a smaller one.
  \end{remark}

  We know that the module $\Lambda_{\Pi}$ in which our Mazur--Tate elements are valued has rank 1 over $\cO_F$. We may thus choose an $\cO_E$-basis $\Omega_{\Pi}$ of $\cO_E \otimes_{\cO_F} \Lambda_{\Pi}$. Let us write
  \[ \mu_{\alpha, \alpha} = \frac{L_{p, (\alpha, \alpha)}(\Pi)}{\Omega_{\Pi}} \in D^{(1/2, 1/2)}(G_{p^\infty}, E)
  \]
  and similarly for $\mu_{\alpha, \beta}$, $\mu_{\beta, \alpha}$, $\mu_{\beta, \beta}$.

  \begin{proposition}
   The distributions
   \begin{align*}
    \mu_{+, +} &= \mu_{\alpha, \alpha} + \mu_{\alpha, \beta} + \mu_{\beta, \alpha} + \mu_{\beta, \beta}\\
    \mu_{+, -} &= \mu_{\alpha, \alpha} - \mu_{\alpha, \beta} + \mu_{\beta, \alpha} - \mu_{\beta, \beta}\\
    \mu_{-, +} &= \mu_{\alpha, \alpha} + \mu_{\alpha, \beta} - \mu_{\beta, \alpha} - \mu_{\beta, \beta}\\
    \mu_{-, -} &= \mu_{\alpha, \alpha} - \mu_{\alpha, \beta} - \mu_{\beta, \alpha} + \mu_{\beta, \beta}
   \end{align*}
   all lie in $D^{\left(1/2, 1/2\right)}(G_{p^\infty}, E)$, and have the property that if $\omega$ is a finite-order character of conductor $\fp^{n_{\fp}} \fpb^{n_{\fpb}}$, with both $n_{\fp}, n_{\fpb} > 0$, then $\mu_{*, \circ}$ vanishes at $\omega$ unless $* = (-1)^{n_{\fp}}$ and $\circ = (-1)^{n_{\fpb}}$.
  \end{proposition}

  \begin{proof}
   Clear from the interpolating property of the $\mu$'s.
  \end{proof}

  Thus each of the distributions $\mu_{*, \circ}$, for $*, \circ \in \{+, -\}$, vanishes at three-quarters of the finite-order characters of the $p$-adic Lie group $G_{p^\infty}$.

  We now interpret this statement in terms of divisibility by half-logarithms. Let $\log^+_{\fp}$ be Pollack's half-logarithm on $\cO_{K, \fp}^\times \cong \Zp^\times$ (cf.~\cite{pollack03}) which is a distribution of order 1/2 with the property that $\log^+_{\fp}(\chi)$ vanishes at all characters $\chi$ of $\Zp^\times$ of conductor $p^n$ with $n$ an odd integer (and no others). Similarly, let $\log^-_{\fp}$ be the half-logarithm which vanishes at characters of conductor $p^n$ with $n$ positive and even. We define $\log_\fp^+$ and $\log_{\fp}^-$ as distributions on $G_{p^\infty}$ by pushforward, and similarly $\log^\pm_{\fpb}$.

  \begin{proposition}
   For $*, \circ \in \{+, -\}$, the distribution $\mu_{*, \circ}$ is divisible in $D^{\la}(G_{p^\infty})$ by $\log^{*}_{\fp} \log^{\circ}_{\fpb}$.
  \end{proposition}

  \begin{proof}
   We begin by introducing distributions
   \begin{align*}
    \mu_{+, \alpha} &= \mu_{\alpha, \alpha} + \mu_{\beta, \alpha}\\
    \mu_{+, \beta}  &= \mu_{\alpha, \beta} + \mu_{\beta, \beta}\\
    \mu_{-, \alpha} &= \mu_{\alpha, \alpha} - \mu_{\beta, \alpha}\\
    \mu_{-, \beta}  &= \mu_{\alpha, \beta} - \mu_{\beta, \beta}.
   \end{align*}
   We claim that $\mu_{+, \alpha}$ and $\mu_{+, \beta}$ are divisible by $\log^+_{\fp}$, and $\mu_{-, \alpha}$ and $\mu_{-, \beta}$ by $\log^-_{\fp}$. We prove only the first of these four statements, since the proofs of the other three are virtually identical. If $\chi$ is any finite-order character such that $n_\fp(\chi)$ (the power of $\fp$ dividing the conductor of $\chi$) is positive and odd, then $\mu_{+, \alpha}$ vanishes at $\chi$ by construction.

   Since $\mu_{+, \alpha}$ is in $D^{(1/2, 1/2)}(G_{p^\infty}, E)$, it follows that it must vanish at any character of $G$ whose pullback to $\cO_{K, \fp}^\times$ is of finite order with odd p-power conductor. Hence $\mu_{+, \alpha}$ is divisible by $\log^+_{\fp}$ in $D^{(1/2, 1/2)}(G_{p^\infty}, E)$ (and the quotient clearly has order $(0, 1/2)$).

   We now apply the same argument in the other variable, to show that $\mu_{\alpha, +} = \mu_{\alpha, \alpha} + \mu_{\alpha, \beta}$ is divisible by $\log^+_{\fpb}$, etc. This shows that $\mu_{+, +}$ is divisible by both $\log^+_{\fp}$ and $\log^+_{\fpb}$. These two distributions are coprime (if we regard them as rigid-analytic functions on the character space of $G_{p^\infty}$, the intersection of their zero loci is a subvariety of codimension 2) so we are done.
  \end{proof}

  \begin{corollary}
   There exist four bounded measures
   \[ L_{p}^{+,+}, L_{p}^{+,-}, L_{p}^{-,+}, L_{p}^{-,-} \in \Lambda_F(G_{p^\infty})\]
   such that
   \begin{multline*}
    \frac{L_{p, (\alpha, \alpha)}(\Pi)}{\Omega_\Pi} = \frac{1}{4} \left(
    L_{p}^{+,+} \log^+_{\fp} \log^+_{\fpb}
    + L_{p}^{+,-} \log^+_{\fp} \log^-_{\fpb}\right. \\
    \left.
    + L_{p}^{-,+} \log^-_{\fp} \log^+_{\fpb}
    + L_{p}^{-,-} \log^-_{\fp} \log^-_{\fpb}
    \right)
   \end{multline*}
   and similarly for the other three unbounded $L$-functions.
  \end{corollary}

  \begin{remark}
   If $\Pi$ is the base-change to $K$ of a weight 2 modular form, then $\mu_{\alpha, \alpha}$ and $\mu_{\beta, \beta}$ coincide with the distributions constructed in \cite{kim-preprint}, up to minor differences in the local interpolation factors. Hence the above corollary proves the conjecture formulated in \emph{op.cit.}. However, it does not appear to be possible to construct the signed distributions $L_{p}^{+,+}$ etc solely from $\mu_{\alpha, \alpha}$ and $\mu_{\beta, \beta}$; it seems to be necessary to use $\mu_{\alpha \beta}$ and $\mu_{\beta \alpha}$ as well, and these last two are apparently not amenable to construction via Rankin--Selberg convolution techniques as in \emph{op.cit.}.
  \end{remark}

\providecommand{\MR}[1]{\relax}
\renewcommand{\MR}[1]{~}%
\providecommand{\href}[2]{#2}
\newcommand{\articlehref}[2]{\href{#1}{#2}}

\end{document}